\documentclass[12pt, twoside]{amsart}
\usepackage{amsmath,amsthm,amscd,amssymb}
\usepackage{latexsym}
\usepackage{color,amsrefs}
\usepackage{tikz}\usepackage{tikz-3dplot}

\usepackage[color=green!15,bordercolor=red,linecolor=red!30]{todonotes}

\usepackage{mathtools}
\mathtoolsset{showonlyrefs,showmanualtags} 

\newtheorem{theorem}[equation]{Theorem} 

\newtheorem{definition}[equation]{Definition}

\newtheorem{lemma}[equation]{Lemma}
\newtheorem*{beck}{Beck's Spherical Cap Discrepancy Theorem}
\newtheorem{stol}[equation]{Stolarsky  Invariance Principle}

\usepackage{hyperref} %,hypertexnames=false,colorlinks,[pagebackref]
\hypersetup{
    colorlinks=true,       % false: boxed links; true: colored links
    linkcolor=blue,          % color of internal links
    citecolor=magenta,        % color of links to bibliography
    filecolor=magenta,      % color of file links
    urlcolor=cyan           % color of external links
}

\numberwithin{equation}{section}

\begin{document}

\begin{abstract}
A sign-linear one bit map from the $ d$-dimensional sphere $ \mathbb S ^{d}$ to the  $ N$-dimensional 
Hamming cube $ H^N= \{ -1, +1\} ^{n}$ is given by 
\begin{equation*}
x \mapsto  \{   \textup{sign} (x \cdot z_j) \;:\; 1\leq j \leq N\}
\end{equation*}
where $ \{z_j\} \subset \mathbb S ^{d}$.    
For $ 0 < \delta < 1$, we estimate $ N (d, \delta )$, the smallest integer $ N$ so that there is a sign-linear map 
 which has the  $ \delta $-restricted isometric property, where we impose normalized geodesic distance on 
$ \mathbb S ^{d}$  and  Hamming metric on $ H^N$. 
Up to a polylogarithmic factor, $ N (d, \delta ) \approx \delta^{-2 + \frac2{d+1}}$, 
which has a dimensional correction in the power of $ \delta $.  
This is a question that 
arises from the one bit sensing literature, and the method of proof follows from 
geometric discrepancy theory. 
We also obtain an analogue of the Stolarsky invariance principle for this situation, which implies that minimizing the $L^2$ average of the  embedding error is equivalent to minimizing the discrete energy $\sum_{i,j} \big( \frac12 - d(z_i,z_j) \big)^2$, where $d$ is the normalized geodesic distance.
\end{abstract}

\title[One Bit Sensing and Discrepancy]{One Bit Sensing, Discrepancy,  and  Stolarsky Principle.}

% \subjclass[2000]{Primary: 42B20 Secondary: 42B25, 42B35}
% \keywords{}
\author{Dmitriy Bilyk}   %  can use \and  

\address{School of Mathematics, University of Minnesota, Minneapolis MN 55455, USA}
\email {dbilyk@math.umn.edu}
\thanks{Research supported in part by  NSF grants DMS 1101519 and DMS 1265570.}

\author{Michael T. Lacey}   %  can use \and  

\address{ School of Mathematics, Georgia Institute of Technology, Atlanta GA 30332, USA}
\email {lacey@math.gatech.edu}
%\thanks{Research supported in part by the NSF grant  DMS 1265570. }
\maketitle

\section{Introduction}

The present paper is concerned with the following question: what is the minimal number of hyperplanes  such that, for any two points on the unit sphere, the geodesic distance between them is well-approximated by the proportion of hyperplanes, which separate these points.
%We approximate, above and below,  the minimal number of hyperplanes required to guarantee that for any two points $x,y\in \mathbb S^d$ the proportion of hyperplanes that separates these points approximates the distance between the points up to a given threshold.  
This question has  connections to  different topics, such as one-bit sensing (a non-linear variant of compressive sensing),   geometric functional analysis (almost isometric embeddings), and combinatorial geometry (tesellations of the sphere), while our proof techniques are taken from geometric discrepancy theory.

%\subsection{Tessellations of the sphere} %One example of such a problem is {\it{tessellation of the sphere by hyperplanes}}. The setup is the following. 

We now introduce the notation and make this question more precise. Let $d\ge2$ and let $\mathbb S^d \subset \mathbb R^{d+1} $ denote the $d$-dimensional unit sphere. We denote by $d(x,y)$  the {\it{geodesic}} distance  between $x$ and $y$ on $\mathbb S^{d}$ normalized so that the distance between antipodal points  is $1$, i.e. 
\begin{equation}\label{e.dgeod}
d(x,y) = \frac{\cos^{-1} ( x\cdot y) }{\pi}.
\end{equation}
The  $ N$-dimensional 
Hamming cube $ H^N= \{ -1, +1\} ^{N}$ has the Hamming metric 
\begin{equation*}
d _{H} (s,t) = \frac 1 {2N} \sum_{j=1} ^{N}   \lvert  s_j - t_j\rvert   =\, \frac1{N} \cdot  \# \{  1\le j \le N:\, s_j \neq t_j \},
\end{equation*}
where $ s = (s_1 ,\dotsc, s_N) \in H^N$, and similarly for $ t$, i.e. $d_H (s,t)$ measures the proportion of the coordinates in which $s$ and $t$ differ.    We consider \emph{sign-linear} maps 
from $ \mathbb S ^{d}$ to $ H ^{N}$ given by 
\begin{equation*}
\varphi _{Z} (x) = \{  \textup{sgn} ( z_j \cdot  x ) \;:\; 1\leq j \leq N \}
\end{equation*}
where $ Z=\{ z_1, z_2 ,\dotsc,  z_N\}\subset \mathbb S ^{d}$.  Note that, with abuse of notation,  $ \varphi_Z (x) = \textup{sgn} (Ax)$,  where the rows of $ A$ consist of the vectors $ z_1 ,\dotsc, z_N$. 

Each coordinate of the map  $ \varphi _{Z}$
divides $ \mathbb S ^{d}$ into two hemispheres, and the Hamming distance 
\begin{equation*}
d _{H} (\varphi _{Z} (x), \varphi _{Z} (y))
\end{equation*}
is the proportion of of hyperplanes  $z_j^\perp$  that {\it{separate}} the points $x$ and $y$.  
It is  easy  to see  that, if one chooses a hyperplane $z^\perp$ uniformly at random, then 
\begin{equation}\label{e.crof}
 \mathbb P \big\{ \textup{sgn} (x\cdot z) \neq  \textup{sgn} (y\cdot z ) \big\} =  {d(x,y)}.
 \end{equation}%where $d(x,y)$ is the   
This is the founding instance of the Crofton formula from integral geometry \cite{MR2168892}*{p. 36--40}.
Hence for a large number of random (or carefully chosen deterministic) hyperplanes, the Hamming distance $d_H (x,y) $ should be close to the geodesic distance $d(x,y)$.

The closeness is quantified by the following  definition of the \emph{restricted isometric property (RIP)}, 
a basic concept in compressed sensing literature. 

%%%%%%%%%%%%%%%%%%%%%%%%%%%%%%  DEFINITION DEFINITION DEFINITION
\begin{definition}\label{d:rip}  Let $ 0< \delta < 1$. We say that $ \varphi : \mathbb S ^{d} \mapsto H ^{N}$ 
satisfies  \emph{$ \delta $-RIP}    if %there holds 
\begin{equation} \label{e.deltauni}
\sup _{x,y \in \mathbb S ^{d}}  \lvert   d _{H} (\varphi  (x), \varphi(y))  - d (x,y)\rvert  < \delta .  
\end{equation}

\noindent We  set   $N(d,\delta)$  to be  the minimal integer $N$, for which   there exists an $N$-point set $Z\subset \mathbb S^d$,  
such that   $ \varphi _{Z}$ is a $ \delta $-RIP map.   
\end{definition}
%%%%%%%%%%%%%%%%%%%%%%%%%%%%%%  DEFINITION DEFINITION DEFINITION

In the sign-linear case  $ \varphi = \varphi _Z$, we set 
\begin{equation}\label{Delta}
\Delta _{Z} (x,y) = d _{H} (\varphi_Z  (x), \varphi_Z (y))  - d (x,y). 
\end{equation}

Building intuition, we can set   $N _{\textup{rdm}}(d,\delta)$ to be the smallest integer $ N$ so that 
drawing $ Z$ uniformly at random, the 
sign-linear map $ \varphi _{Z}$ is a $ \delta $-RIP map, with chance at least $ 1/2$.    In a companion paper \cite{BL}, we 
conjecture, following \cite{MR3164174}, that 
\begin{equation} \label{e:random}
N _{\textup{rdm}}(d,\delta) \lesssim  d \delta ^{-2}. 
\end{equation}
Such bounds are known  for the \emph{linear} embedding of the sphere into $\mathbb R^N$ (Dvoretsky theorem).  The power of $ \delta ^{-2}$ is sharp in the random case, which follows from the Central Limit Theorem.  In \cite{BL} we prove that a $\delta$-RIP map from $\mathbb S^d$ to $\mathcal H^N$ exists for $N$ as in \eqref{e:random}, although our map is not sign-linear, but rather a composition of the ``nearest neighbor" map and a sign-linear map. We also prove an analog of \eqref{e:random} for sparse vectors. 

%For 
In this paper %, we concentrate on the entire sphere, and 
we show that in general there is a  dimensional correction to the 
 power of  $ \delta $.   This is our first main result.  

%%%%%%%%%%%%%%%%%%%%%%%%%%%%%% THEOREM THEOREM THEOREM
\begin{theorem}\label{t:polylog} For all $ d \in \mathbb N $ and $ 0 < \delta  <1 $, there holds, 
\begin{equation}\label{e.ndelta}
     N (d, \delta)  \approx  _{\textup{log}}\delta^{-2 + \frac{2}{d+1}}  .  
\end{equation} 
where the equality holds up to a dimensional constant and a polylogarithmic factor in $ \delta $. 
\end{theorem}
%%%%%%%%%%%%%%%%%%%%%%%%%%%%%% THEOREM THEOREM THEOREM

The upper bound in \eqref{e.ndelta} is achieved by exhibiting a $ Z$ of small cardinality, for which $ \varphi _Z $ satisfies  $ \delta $-RIP.  {\emph{Jittered (or stratified) sampling}}, 
a cross between purely random and deterministic constructions, 
provides the example.   Loosely speaking, first one divides the sphere $\mathbb S^d$ into $N$ roughly equal pieces, and then chooses a random point in each of them, see \S\ref{s.js} for details.  The lower bound is the universal statement that every $ Z$ of sufficiently small cardinality does not yield  a $ \delta $-RIP map. It is a deep fact from geometric discrepancy theory. 

Most of the prior work concerns randomly selected $ Z$. 
Jacques and coauthors \cite{MR3043783}*{Thm 2}  proved an analog of  \eqref{e:random}  sparse vectors in $ \mathbb S ^{d}$ with an additional logarithmic 
term in $ \delta $.  %In fact, they also considered the sparse vectors in $ \mathbb S ^{d}$. 
Plan and Vershynin \cite{MR3164174} studied this question, looking for RIP's for general subsets $ K\subset \mathbb S ^{d}$ 
into the Hamming cube, of which the sparse vectors are a prime example.  
They proved  \cite[Thm 1.2]{MR3164174} that $ N _{\textup{rdm}} (d, \delta ) \lesssim  d \delta ^{-6} $, and conjectured 
\eqref{e:random}, at least in the random case.   
Neither paper   anticipates the dimensional correction in $ \delta $ above.

Since in applications the dimension $d$ is often quite large, we considered a non-asymptotic version of the upper bound in \eqref{e.ndelta} and computed an effective value of the constant $C_d$, proving that it grows roughly as $d^{5/2}$ (see Theorem \ref{t.natess} for a more precise statement): 
\begin{equation}\label{e.ndelta1}
N(d, \delta) \le \max \Big\{ C d^{5/2}  \, \delta^{-2 + \frac2{d+1}} \, \Big( 1 + \log d + \log \frac1{\delta} \Big)^\frac{d}{d+1},\,\,\, 100d \Big\},
\end{equation}
where $C>0$ is an absolute constant.%, which can be taken equal to $C= 2400$. \\

\smallskip 
Our second main result goes in a  somewhat different direction. In Theorem \ref{t.stol} we show that the $L^2$ norm of $\Delta_Z(x,y)$ given in \eqref{Delta} satisfies an analog of the Stolarsky principle \cite{MR0333995}, which implies that minimizing the $L^2$ average of $\Delta_Z (x,y) $ is equivalent to minimizing the discrete energy of the form $\frac1{n^2} \sum_{i,j} \big( \frac12 - d(z_i,z_j) \big)^2$. 
This suggests interesting connections to such objects as spherical codes, equiangular lines, and frames. See Theorem \ref{t.stol} and \S\ref{s.stol} for details. \\

%%%%%%%%%%%%%%%%%%%%%%%%%%%%%% SUBSECTION SUBSECTION SUBSECTION SUBSECTION
 %%%%%%%%%%%%%%%%%%%%%%%%%%%%%% SUBSECTION SUBSECTION SUBSECTION SUBSECTION 
\subsection*{One-Bit Sensing}%\label{ss.}
The restricted isometry property (RIP)  was formulated by Candes and Tao \cite{MR2300700}, 
and is a basic concept in the compressive sensing \cite{MR3100033}*{Chap. 6}.  
It can be studied in various  metric spaces, and thus has many interesting variants. 

One-bit sensing was initiated by Boufounos and Baraniuk \cite{1Bit}. The motivation for the one bit measurements 
$ \textup{sgn} (x \cdot y)$ are that (a) they form a canonical non-linearity on the measurement, as well as a 
canonical quantization of data, 
(b) there are 
striking technological advances which employ non-linear observations, and (c) it is therefore of interest to develop 
a comprehensive theory of non-linear signal processing.  

The subsequent theory has been developed by \cites{13051786,MR3043783,MR3069959,MR3164174}.  
For upper bounds, random selection of points on a sphere is generally used.  Note that 
\cite{13051786}*{Thm 1} does contain a lower bound on the rate of recovery of a one-bit decoder.  
Plan and Vershynin \cite{MR3164174}  established results  on one-bit RIP maps for 
arbitrary subsets of the unit sphere, and proposed some ambitious conjectures about bounds for these maps.  
In a companion paper \cite{BL} we will investigate some of these properties in the case of randomly selected hyperplanes.  
The results about one-bit sensing have been used in other interesting contexts, see  the papers cited above as well as 
\cites{MR3008160,14078246}.   

Lower bounds,  like the ones  proved in  Theorem \ref{t:polylog},  indicate the limits of what can be accomplished 
in compressive sensing.  
See for instance Larsen--Nelson \cite{14112404} who prove a \emph{lower bound} for dimension reduction in the  Johnson--Lindenstrauss 
Lemma.    This lemma is a foundational result in dimension reduction.  
In short, it states that for $ X\subset  \mathbb S^d \subset \mathbb R ^{d+1}$ of cardinality $ k$, 
there is a linear map $ A :\; \mathbb R ^{d+1} \mapsto \mathbb R ^{N}$, which, restricted to $ X$, satisfies  $ \delta $-RIP, 
provided $ N \gtrsim  \delta ^{-2}\log k$.   This has many proofs, see for instance \cite{MR2199631}.  The connection of this lemma to compressed sensing is well known, see e.g. \cite{MR2453366}.

In our companion paper \cite{BL}, we will show that the \emph{one-bit} variant of the 
Johnson--Lindenstrauss bound holds, with the same 
bound  $ N \gtrsim  \delta ^{-2}\log k$.  It would be interesting to know if this bound is also sharp. 
The clever techniques of \cite{14112404} are essentially linear in nature, so that a new technique is needed. 
Progress on this question could have consequences on lower bounds for \emph{non-linear} Johnson--Lindenstrauss 
RIPs.  

%%%%%%%%%%%%%%%%%%%%%%%%%%%%%% SUBSECTION SUBSECTION SUBSECTION SUBSECTION
 %%%%%%%%%%%%%%%%%%%%%%%%%%%%%% SUBSECTION SUBSECTION SUBSECTION SUBSECTION 
\subsection*{Dvoretsky's Theorem}%\label{ss.}
The results of this paper are also related to Dvoretsky's Theorem \cite{MR0139079}, which states that 
for all $ \epsilon >0$ and all  dimensions $ d$, there  exists $ N = N (d, \epsilon )$ so that any Banach space $ X$ 
of dimension $ N$ contains a subspace $ Y$ of dimension $ d$ which embeds into Hilbert space with distortion 
at most $ 1 + \epsilon $.  (Finite distortion must hold uniformly at all scales, in contrast to the RIP, which 
    ignores sufficiently small scales.)  
This is a fundamental result in geometric functional analysis, and has sophisticated variants in metric spaces, 
\cites{MR2995229,MR3087345}.   

It is interesting that the argument of Plan and Vershynin \cite{MR0139079}*{\S3.2} relies upon a variant of Dvoretsky's Theorem 
and indeed ties improved bounds in Dvoretsky's Theorem to improvements in one-bit RIP maps.  In view of the connection 
between RIP properties in geometric discrepancy identified in this paper, there are new techniques that could be brought 
to bear on this question.  

%%%%%%%%%%%%%%%%%%%%%%%%%%%%%% SUBSECTION SUBSECTION SUBSECTION SUBSECTION
 %%%%%%%%%%%%%%%%%%%%%%%%%%%%%% SUBSECTION SUBSECTION SUBSECTION SUBSECTION 
\subsection*{Geometric Interpretation}%\label{ss.}
The results above can be interpreted as properties of  tessellations of the sphere $ \mathbb S ^{d}$ induced 
by the hyperplanes $ \{ z ^{\perp} \;:\; z\in Z\}$. The integer $ N (d, \delta )$ is the smallest size of 
$ Z$ so that for all $ x,y \in \mathbb S ^{d}$, the proportion of hyperplanes from $ Z$ that separate  $ x,y$ 
is bounded above and below by $ d (x,y) \pm \delta $.  This is the geometric language used in 
Plan--Vershynin \cite{MR0139079}, which indicates a connection with geometric discrepancy theory.

We point the reader to some  recent papers which investigate integration on spheres and related geometrical 
questions: \cites{150403029,MR3000572,MR3365840}

\bigskip 
In this paper, we shall denote by $\sigma$ the surface measure on the sphere, normalized so that $\sigma(\mathbb S^d) =1$. Unnormalized (Hausdorff) measure on $\mathbb S^d$ will be denoted by $\sigma^*_d$. We shall use the notation $\omega = \sigma^*_{d-1} (\mathbb S^{d-1})$ and  $\Omega = \sigma^*_{d} (\mathbb S^{d})$.  In particular 
\begin{equation*}
\Omega = \frac {2 \pi ^{\frac {d+1}2}} {\Gamma (\frac {d+1}2)},
\end{equation*}
and the ratio between these two, which will appear often,  satisfies (see \cite{MR2582801})
\begin{equation} \label{e:ratio}
\frac \omega \Omega = \frac {\Gamma (\frac {d+1}2)} {{\Gamma (\frac {d}2)} \sqrt \pi } \leq \sqrt { \frac d {2 \pi }}.  
\end{equation}
The notation $ A \lesssim B$ means that $ A \le CB$ for some fixed constant $C>0$. Occasionally, the implicit constant may depend on the dimension $d$ (this will be made clear in the context), but it is always  independent of $N$ and $\delta$.

\subsection{Discrepancy}  
We phrase the RIP property in the language of  geometric discrepancy theory  on the sphere $\mathbb S^{d}$.
 Let $Z=\{z_1 ,\dotsc, z_N \}$ be an $N$-point subset of $\mathbb S^d$. 
 The discrepancy of $ Z$ relative to a measurable subset $ S\subset \mathbb S ^{d}$ is 
\begin{equation}
D(Z,S) = \frac1{N} \cdot \# \{ Z \cap S \}  - \sigma (S).
\end{equation}
 We define the extremal ($ L ^{\infty }$) discrepancy of $Z$ with respect to a family $\mathcal S$ of measurable subsets of $ \mathbb S ^{d}$ to be 
\begin{equation}\label{e.discdef}
D_{\mathcal S} (Z) = \sup_{S\in \mathcal S} \big|D(Z,S)\big|.
\end{equation}
If the  family $\mathcal S$ admits a natural measure then one may also replace the supremum above by an $L^2$ average.  The main questions of discrepancy theory are: How small can discrepancy be? What are good or optimal point distributions? 
These questions have profound connections to approximation theory, probability, combinatorics, number theory, computer science, analysis, etc, see \cite{MR903025,MR1779341,MR1697825}.

In this sense  the quantity $\Delta_Z (x,y)$ defined in \eqref{Delta} clearly has a discrepancy flavor.
In fact, (and this is perhaps the most important  observation of the paper) the problem of uniform tessellations  can actually be reformulated as a problem on geometric discrepancy with respect to  \emph{spherical wedges}. 

Denote   the set of normals of those hyperplanes that separate $x$ and $y$ by 
\begin{equation}\label{e.wedge}
 W_{xy} = \big\{ z \in \mathbb S^{d}:\, \textup{sgn} (x\cdot z) \neq  \textup{sgn} (y\cdot z )  \big\}.
 \end{equation}
The letter $W$ stands for {\emph{wedge}}, since the set $W_{xy}$  does in fact look like a spherical wedge, i.e. the subset of the sphere lying between the hyperplanes  $x\cdot z = 0$ and $y\cdot z = 0$, see Figure \ref{f:W}.
 
 %%%%%%%%%%%%%%% Figure
\begin{figure}
\tdplotsetmaincoords{65}{110}

\pgfmathsetmacro{\rvec}{1.2}

\pgfmathsetmacro{\thetavecc}{55}
\pgfmathsetmacro{\phivecc}{35}

\pgfmathsetmacro{\thetaveccc}{39.7}
\pgfmathsetmacro{\phiveccc}{55}

\begin{tikzpicture}[scale=2,tdplot_main_coords]

\shadedraw[tdplot_screen_coords,ball color = white] (0,0) circle (\rvec);

%-----------------------
\coordinate (O) at (0,0,0);

\tdplotsetcoord{B}{\rvec}{\thetavecc}{\phivecc}

\tdplotsetcoord{C}{\rvec}{\thetaveccc}{\phiveccc}

%draw the main coordinate system axes
\draw[thick,->] (0,0,0) -- (1.05,-.6,0) node[anchor=north east]{$x$};
\draw[thick,->] (0,0,0) -- (1.2, -0.2,0) node[anchor=north east]{$y$};
%\draw[thick,->] (0,0,0) -- (0,0,1.7) node[anchor=south]{$z$};

\draw (0,.52,-0.15) node {$ \small{W _{xy}}$}; 

%\draw[-stealth,very thick,color=blue] (O) -- (B);
%

%\draw[-stealth,very thick,color=green!60!black] (O) -- (C);

%\draw[dashed, color=blue] (O) -- (Bxy);
%\draw[dashed, color=blue] (B) -- (Bxy);
%\draw[dashed, color=green!60!black] (O) -- (Cxy);
%\draw[dashed, color=green!60!black] (C) -- (Cxy);

%\tdplotdrawarc[color=blue]{(O)}{0.3}{0}{\phivecc}{anchor=north}{$\lambda_A$}

\tdplotsetthetaplanecoords{\phivecc}

%\tdplotdrawarc[color=blue,tdplot_rotated_coords]{(0,0,0)}{0.3}{90}{\thetavecc}{anchor=south west}{$\varphi_A$}

%\tdplotdrawarc[color=green!40!black]{(O)}{0.7}{0}{\phiveccc}{anchor=north}{$\lambda_B$}

\tdplotsetthetaplanecoords{\phiveccc}

%\tdplotdrawarc[color=green!40!black,tdplot_rotated_coords]{(0,0,0)}{0.7}{90}{\thetaveccc}{anchor=south west}{$\varphi_B$}
\draw[dashed] (\rvec,0,0) arc (0:360:\rvec);
\draw[thick] (\rvec,0,0) arc (0:110:\rvec);
\draw[thick] (\rvec,0,0) arc (0:-70:\rvec);

\tdplotsetthetaplanecoords{35}
\draw[thick,tdplot_rotated_coords] (\rvec,0,0) arc (0:151:\rvec);
%\draw[very thick,color=red,tdplot_rotated_coords] (\rvec,0,0) arc (0:55:\rvec);
\draw[dashed,tdplot_rotated_coords] (\rvec,0,0) arc (180:-40:-\rvec);
\draw[thick,tdplot_rotated_coords] (\rvec,0,0) arc (360:336:\rvec);

\tdplotsetthetaplanecoords{55}
\draw[thick,tdplot_rotated_coords] (\rvec,0,0) arc (0:147:\rvec);
%\draw[very thick,color=red,tdplot_rotated_coords] (\rvec,0,0) arc (0:40:\rvec);
\draw[dashed,tdplot_rotated_coords] (\rvec,0,0) arc (180:-40:-\rvec);
\draw[thick,tdplot_rotated_coords] (\rvec,0,0) arc (360:334:\rvec);

\tdplotsetrotatedcoords{-79.1}{-120}{27.3}
%\draw[very thick,color=red,tdplot_rotated_coords] (\rvec,0,0) arc (0:21:\rvec);

\end{tikzpicture}

\caption{The spherical  wedge $ W _{xy}$. }
\label{f:W}
\end{figure}
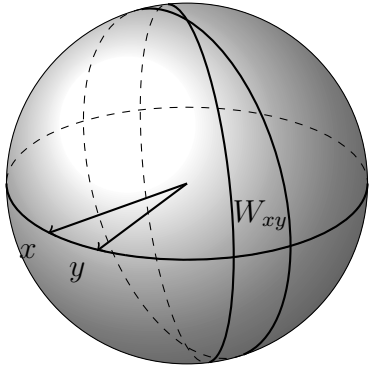
%%%%%%%%%%%%%%% Figure

%Let $\sigma$ denote the {\emph{normalized} surface measure on the sphere $S^{d-1}$, i.e. ${\int\limits_{S^{d-1}}} d\sigma = 1$. 
It follows from  the Crofton formula \eqref{e.crof} that 
\begin{equation}
\sigma (W_{xy}) =  \mathbb P ( z^\perp \textup{ separates } x \textup{ and } y ) =  d(x,y).
\end{equation}
Therefore we can rewrite the quantity \eqref{Delta} as 
\begin{equation}\label{Delta1}
\Delta_Z (x,y) = \frac{\#  (Z\cap W_{xy}) }{N} \ - \sigma (W_{xy})\, = \, \frac{1}{N} \sum_{k=1}^N {\bf 1}_{W_{xy} } (z_k)  - \sigma (W_{xy}) \eqqcolon D \big(Z, W_{xy}\big),
\end{equation} 
i.e. the  discrepancy of the $N$-point distribution $Z$ with respect to the wedge $W_{xy}$, see \S\ref{s.mr}.

The RIP property can now  be reformulated  in terms of the  $L^\infty$ discrepancy with respect to wedges. Indeed, according to definitions \eqref{e.deltauni} and \eqref{e.discdef}, the map $ \varphi _{Z}$ is  $ \delta $-RIP 
exactly when the quantity below is at most $ \delta $:
\begin{equation}\label{Delta11}
\big\|  \Delta_Z   \big\|_\infty  = \sup_{x,y \in \mathbb S^{d}}  \left| \frac{\#  (Z\cap W_{xy}) }{N} \ - \sigma (W_{xy}) \right| \eqqcolon  D_{\textup{wedge}} (Z) .
\end{equation}
The problem of estimating $N(d,\delta)$ is thus simply inverse to obtaining discrepancy estimates in terms of $N$, and this is precisely the approach we shall take.

%The authors of \cite{MR3164174} considered an inverse problem: what is the minimal size of $Z$ so that $\Delta (Z) < \delta$, in other words, what is the minimal number $N$ of hyperplanes so that the Hamming distance approximates the geodesic distance uniformly up to $\delta$. They conjectured that $N \approx \delta^{-2}$.\\

%It appears however that this conjecture is, in general, not true. While the {\emph{theory of empirical processes}} suggests that this should be the right relation for {\bf{random}} tessellations by hyperplanes, {\emph{discrepancy theory}} seems to suggest that in the general case there should be a dimensional correction.\\

\subsection{A point of reference: spherical cap discrepancy}

We recall the classical results concerning the discrepancy for spherical caps.
For $x\in \mathbb S^d$, and $t\in [-1,1]$, let $C(x,t)$ be the spherical cap  of height $t$ centered at $x$,  given by 
\begin{equation}
C(x,t) = \{ y\in \mathbb S^d:\, y\cdot x \ge t\}.
\end{equation}
Denote the set of all spherical caps by $\mathcal C$. For an $N$-point set $Z \subset \mathbb S^d$ let
\begin{equation}
D_{\textup{cap}} (Z) = \sup_{C\in \mathcal C }   \big| D \big(Z, C \big)  \big| =  \sup_{C\in \mathcal C }     
\Bigl\lvert \frac{\#  (Z\cap C ) }{N} \ - \sigma (C)  \Bigr\rvert
\end{equation} 
be the extremal  discrepancy of $Z$ with respect to spherical caps $\mathcal C$. The following classical results  due to J. Beck \cite{MR736726,MR762175} yield almost precise information about the growth  of this quantity in terms of $N$.

\begin{beck}\label{l1}%[Beck's Theorem on Spherical Cap Discrepancy] 
 For dimensions $ d \geq 2$, there holds 
%%  ENUMERATE
\begin{description}
\item[Upper Bound] There exists an $N$-point set $Z\subset \mathbb S^{d}$ with spherical cap discrepancy 
\begin{equation}\label{e.t1}
D_{\textup{cap}} (Z) \lesssim N^{-\frac12 - \frac{1}{2{d}}} \sqrt{\log N}.
\end{equation} 

\item[Lower Bound]  For any $N$-point set $Z\subset \mathbb S^{d}$  the spherical cap discrepancy satisfies
\begin{equation} \label{e:lower}
D_{\textup{cap}} (Z) \gtrsim  N^{-\frac12 - \frac{1}{2{d}}}.
\end{equation} 
\end{description}
%% ENUMERATE
\end{beck}

 We will elaborate on the upper bound \eqref{e.t1}.  
It  is proved  using a construction known as \emph{jittered sampling}, which produces  a semi-random point set. We describe this construction in much detail  in \S \ref{s.js}. The proof of \eqref{e.t1} in  \cite{MR762175} states that ``...using probabilistic ideas it is not hard to show..." and refers to   \cite{MR736726}, where this fact is proved for rotated rectangles, not spherical caps.  (It is well known that jittered sampling is  applicable in many geometric settings.)
In the book \cite{MR903025} the algorithm is described in some more detail, but one of the key steps, namely regular 
equal-area partition of the sphere, is only postulated. This construction was only recently rigorously formalized and  effective values of the underlying constants have been found \cites{MR1900615,MR2582801}. See \S \ref{s.rps} for further discussion.

It is generally believed that standard low-discrepancy sets, while providing  good bounds with respect to the number of points $N$, yield very bad, often exponential,  dependence on the dimension. However, as we shall see, it appears that for jittered sampling this behavior is quite reasonable (see also \cite{151000251} for a discussion of a similar effect).  This is consistent with the fact that  this construction is intermediate between purely random and deterministic sets. 

Since we are interested  both in asymptotic and non-asymptotic regimes, we shall explore this construction (in the case of spherical wedges) very scrupulously tracing the dependence of the constant on the dimension.

The proof of the lower bound \eqref{e:lower} is Fourier-analytic in nature and  holds with the smaller  $L^2$ average in place of the supremum, 
\begin{equation}
D_{\textup{cap}, L^2} (Z) = \left( \int_{-1}^1 \int_{\mathbb S^d}    
\Bigl\lvert \frac{\#  \big(Z\cap C(x,t) \big) }{N} \ - \sigma \big(C(x,t)\big)  \Bigr\rvert
^2 d\sigma(x) \,dt \right)^{1/2}. 
\end{equation}
More precisely, a   lower bound stronger than \eqref{e:lower} holds:  
\begin{equation}\label{e.BeckL2}
D_{\textup{cap},L^2} (Z) \gtrsim  N^{-\frac12 - \frac{1}{2{d}}}.
\end{equation}
This bound is sharp:  
the $L^2$ discrepancy of   jittered sampling yields a bound akin to \eqref{e.t1}, but without $\sqrt{\log N}$.

Strikingly,  minimizing the $ L ^2 $ discrepancy is the same as maximizing  the sum of pairwise distances between 
the vectors in $ Z$, which is the main result of  \cite{MR0333995}.   

\begin{stol}\label{t.originalstol}%[Stolarsky  Invariance Principle]
In all dimensions $d\ge 2$, for any $N$-point set $Z = \{z_1,\dotsc,z_N \} \subset \mathbb S^d$, the following holds 
\begin{equation} \label{e:stol}
\frac1{c_d} \big[ D_{\textup{cap}, L^2} (Z)  \big]^2 =   \int_{\mathbb S^{d}} \int_{\mathbb S^{d}} \|  x- y \| \, d\sigma (x)\, d\sigma (y) - \frac{1}{N^2} \sum_{i,j = 1}^N \| z_i - z_j \| ,
\end{equation}
where $\| \cdot \|$ is the Euclidean norm and $c_{d} = \frac12 \int_{\mathbb S^{d}} |p \cdot z| d\sigma (z) = \frac{1}{d} \frac{\omega}{\Omega}$ for an arbitrary pole $p\in \mathbb S^d$.
\end{stol}

The square of the $L^2$ discrepancy  is exactly  the difference between the continuous potential  energy given by $ \lVert x-y\rVert$ and the discrete energy induced by the points of $Z$.   Alternate proofs of Stolarsky principle   can be found in \cite{MR3034434,B}.

\subsection{Main results}\label{s.mr}

Analogues of Beck's discrepancy estimates \eqref{e.t1} and \eqref{e:lower}, as well as of the  Stolarsky invariance principle hold for spherical wedges $W_{xy}$. These in turn imply results for sign-linear RIP maps.  Moreover, we shall explore the dependence of the upper estimates on the dimension $d$.
Recall the definition of a wedge $ W _{xy}$ in \eqref{e.wedge} and of the wedge discrepancy  \eqref{Delta1}-\eqref{Delta11}. 
\begin{equation*}
D_{\textup{wedge}} (Z) = \sup_{x,y\in \mathbb S ^{d} }   \big| D \big(Z, W _{xy} \big)  \big| =  
\sup_{x,y\in \mathbb S ^{d} }
\Bigl\lvert \frac{\#  (Z\cap W _{xy} ) }{N} \ - \sigma (W _{xy}) \Bigr\rvert. 
\end{equation*}

\begin{theorem}\label{t.wupper} 
For all integers $ d \geq 2$,  there are  $ B_d, \, C_d >0$ so that for all  integers $ N\geq 1$, 
%%  ENUMERATE
\begin{description}
\item[Upper Bound] There exists a distribution of $N$ points $Z  \subset \mathbb S^d$ with 
\begin{equation}\label{e.wupper}
D_{\textup{wedge}}(Z) \le C_d N^{-\frac12 - \frac{1}{2d}} \sqrt{\log N}.
\end{equation}
Provided $ N \ge 100 d$, we have $ C_d \leq  20 d^{\frac34 +\frac1{4d}}  $. 

\item[Lower Bound]  For any $ Z \subset \mathbb S ^{d}$ with cardinality $ N$, 
\begin{equation}\label{e.wlower}
D_{\textup{wedge}} (Z) \ge B_d N^{-\frac12 - \frac{1}{2d}}. 
\end{equation}
\end{description}
%% ENUMERATE

\end{theorem}
%A similar result %(with different absolute constants, but the same exponents of $N$) 
%holds for  the implicit constants in Beck's  spherical cap discrepancy estimate \eqref{e.t1}. 

Both inequalities are known in a very similar geometric situation. It was proved by Bl\"umlinger \cite{MR1116689} that the upper bound \eqref{e.wupper} holds for the discrepancy with respect to spherical ``slices". For $x$, $y\in \mathbb S^d$ denote 
\begin{equation}
S_{xy}  = \{ z \in \mathbb S^d :\, z\cdot x >0,\, z\cdot y <0 \}.
\end{equation}    
In other words, the slice $S_{xy}$ is a half of the wedge $W_{xy}$. It should be noted that the discrepancy with respect to slices is in fact  a better measure of equidistribution on the sphere  than the wedge discrepancy (the wedge discrepancy doesn't change if we move all points to the hemisphere $\{ x\cdot p \ge 0 \}$ by changing some points $x$ to $-x$). Using jittered sampling in a manner almost identical  to Beck's, Bl\"umlinger showed that there exists $Z\subset \mathbb S^d$, $\# Z = N$, such that
\begin{equation}\label{e.ublu}
D_{\textup{slice}} (Z) = \sup_{x,y \in \mathbb S^d} \big| D (Z, S_{xy}) \big| \lesssim N^{-\frac12-\frac{1}{2d}} \sqrt{\log N}.
\end{equation}
Like Beck's estimate, this  bound did not say anything about the dependence of constants on the dimension. Without any regard for constants, the main  estimate \eqref{e.wupper} of Theorem \ref{t.wupper} follows immediately since 
\begin{equation}
D (Z, W_{xy} ) = D (Z, S_{xy}) + D(Z, S_{-x,-y}), \textup{ and hence }\,\,  D_{\textup{wedge}} (Z) \le 2 D_{\textup{slice}} (Z).
\end{equation}
An effective value of the constant $C_d$ in Theorem \ref{t.wupper}, which is important for uniform tessellation and one-bit compressed sensing problems, requires much more delicate considerations and constructions, some of which only became available recently, see \S\ref{s.rps}.

Bl\"umlinger also showed that the lower bound \eqref{e.wlower} holds for the slice discrepancy. The proof uses spherical harmonics and is quite involved (Matou\v{s}ek \cite{MR1697825} writes that `it would be interesting to find a simple proof'). In fact, it is proved that the $L^2$-discrepancy for slices is bounded below by the $L^2$-discrepancy for spherical caps, from which the result follows by Beck's estimate \eqref{e.BeckL2}:
\begin{equation}\label{e.sliceL2}
D_{\textup{slice}} (Z)  \gtrsim  D_{\textup{slice}, L^2} (Z) \gtrsim  D_{\textup{cap},L^2} (Z) \gtrsim  N^{-\frac12 - \frac{1}{2{d}}}.
\end{equation}
The lower bound  for spherical wedges can be deduced by the following symmetrization argument.

%%%%%%%%%%%%%%%%%%%%%%%%%%%%%% PROOF PROOF PROOF
\begin{proof}[Proof of \eqref{e.wlower}]
For a point set $Z = \{ z_1, ..., z_N\} \subset \mathbb S^d $, consider its symmetrization, i.e. a $2N$-point set $Z^* = Z \cup (-Z)$. It is easy to see that 
\begin{equation}
D(Z, W_{xy} ) = D(Z, S_{xy}) + D(Z, S_{-x,-y}) =  D(Z, S_{xy}) +  D( - Z, S_{xy})  =  2 D (Z^*, S_{xy}).
\end{equation}
Therefore, 
\begin{equation}
 D_{\textup{wedge}} (Z) \ge 2 D_{\textup{slice}} (Z^*) \gtrsim (2N)^{-\frac12 - \frac{1}{2{d}}},
\end{equation}
which proves \eqref{e.wlower}. 

\end{proof}
%%%%%%%%%%%%%%%%%%%%%%%%%%%%%% PROOF PROOF PROOF

\vskip5mm

Inverting the bounds of Theorems \ref{t.wupper}, one immediately obtains the first announced result \eqref{e.ndelta},   asymptotic bounds on the minimal dimension of a  sign-linear $\delta$-RIP from $ \mathbb S ^{d}$ to the 
Hamming cube. 

\begin{theorem}\label{t.natess}
There exists an absolute constant $C>0$ (independent of the dimension) such that in every dimension $d\ge 2$ and for every $\delta >0$,  the the integer   $N(d,\delta)$ of Definition~\ref{d:rip} satisfies 
\begin{equation}
b_d \delta^{-2 + \frac2{d+1}}  \leq N (d,\delta)  \le \max  \left\{ 100 d, \,\, C d^\alpha \,\, \delta^{-2 + \frac2{d+1}}  \, \Bigl(  1 + \log d + \log \frac1{\delta}  \Bigr)^{\frac{d}{d+1}} \right\},
\end{equation}
where $\alpha = \frac52 - \frac{2}{d+1}$, and $ b_d >0$.  
\end{theorem} 

The  absolute constant $C$ above can be taken to be e.g. $C=4000$. Some details are given in the end of \S\ref{s.dwjs}.

In a different vein, we also obtain a variant of the Stolarsky Invariance Principle  \eqref{e:stol}.

\begin{theorem}[Stolarsky principle for wedges]\label{t.stol} For any finite set $Z=\{z_1, \dots , z_N \} \subset \mathbb S ^{d}$, the  following relation holds
\begin{equation}\label{e.stol}
\big\| \Delta_Z (x,y) \big\|_2^2  =   \frac{1}{  N^2}  \sum_{i,j=1}^N \bigg(\frac{1}2 -  d (z_i, z_j) \bigg)^2   -   \int\limits_{\mathbb S ^{d}}  \int\limits_{\mathbb S ^{d}}   \bigg(\frac{1}2 -  d (x,y) \bigg)^2  d\sigma (x) \, d\sigma(y)  .
\end{equation}
\end{theorem}

Minimizing the $L^2$ average of the wedge discrepancy associated to the tessellation of the sphere is thus equivalent to minimizing the discrete potential energy of $Z$ induced by the  potential  $P(x,y) = \big( \frac12 - d(x,y) \big)^2$. Intuitively, one would like to make the elements of $Z$ `as orthogonal as possible' on the average.

First of all, this suggests natural candidates for tessellations that are good or optimal on the average, e.g., spherical codes (sets $X\subset \mathbb S^d$ such that  all $x$, $y\in X$ satisfy $x\cdot y < \mu$ for some parameter $\mu \le 1$), see \cite[Chapter 5]{MR2848161} and references therein, or equiangular lines (sets $X\subset \mathbb S^d$ such that  all $x$, $y\in X$ satisfy $|x\cdot y| = \mu$ for some fixed  $\mu \in [0,1)$), see \cite{CasTrem}.

This also brings up connections to frame theory. Benedetto and Fickus \cite{MR1968126} proved that a set $Z= \{z_1,\dotsc, z_N \} \subset \mathbb S^d$  forms a normalized tight frame  (i.e. there exists a constant $A>0$ such that for every $x \in \mathbb R^{d+1}$ an analog of Parseval's identity holds: $A \| x\|^2 = \sum_{i=1}^N | x\cdot z_i |^2  $) if and only if $Z$ is a minimizer of a discrete energy known as the {\emph{total frame potential}}:
\begin{equation}
TP (Z)  = \sum_{i,j=1}^N |z_i \cdot z_j |^2 ,
\end{equation}
which looks unmistakably similar to the discrete energy on the right-hand side of \eqref{e.stol}.

It is not known yet whether the minimizers of \eqref{e.stol} admit a similar geometric or functional-analytic characterization, or if some of the known distributions yield reasonable values for this energy. These are interesting questions  to  be addressed in future research.  We prove Theorem \ref{t.stol} in \S\ref{s.stol}.

\section{Jittered sampling}\label{s.js}

{\emph{Jittered (or stratified)}} sampling in discrepancy theory and statistics can be viewed as a semi-random construction, somewhat   intermediate  between the purely random Monte Carlo algorithms and the purely deterministic low discrepancy point sets. It is easy to describe the main idea in just a few words: initially, the ambient manifold (cube, torus, sphere etc) is subdivided into $N$ regions of equal volume and (almost) equal diameter, then a point is chosen uniformly at random in each of these pieces, independently of others.

Intuitively, this construction guarantees that the arising point set is fairly well distributed (there are no clusters or large gaps). Amazingly, it turns out that in many situations this distribution yields nearly optimal discrepancy (while purely random constructions are far from optimal, and deterministic sets are hard to construct). As mentioned before, the construction in Theorem  \ref{l1} is precisely the jittered sampling. It differs from the corresponding lower bound only by a factor of $\sqrt{\log N}$, which is a result of the application of large deviation inequalities. If one replaces the $L^\infty$ norm of the discrepancy by $L^2$, jittered sampling actually easily gives the sharp upper bound (without $\sqrt{\log N}$). Similar phenomena persist in other situations (discrepancy with respect to  balls or rotated rectangles in the unit cube, slices on the sphere etc).

Jittered sampling is very well described in classical references on discrepancy theory, such as \cite{MR903025,MR1779341,MR1697825}. However, since the spherical case possesses certain subtleties, and, in addition, we want to trace the dependence of the constants on the dimension, we  shall describe the construction in full detail. Besides, this procedure will also yield a quantitative bound on the constant in the classical spherical caps discrepancy estimate \eqref{e.t1}.

In order to make the construction precise we need to introduce two notions: {\emph{equal-area partitions with bounded diameters}} and {\emph{approximating families}}.

\subsection{Regular partitions of the sphere}\label{s.rps}

Let $S_i \subset \mathbb S^{d}$, $i=1,2 ,\dotsc, N$. We say that $\{ S_i \}_{i=1}^N$ is a partition of the sphere if $\mathbb S^{d}$ is a disjoint (up to measure zero) union of these sets, i.e. $\displaystyle{\mathbb S^{d} = \bigcup_{i=1}^N S_i }$ and $\sigma( S_i \cap S_j ) = 0$ for $ i \neq j$.

\begin{definition}
Let  $\mathcal S = \{ S _i\}_{i=1}^N $ be a partition of $\mathbb S^{d}$.  We call it an \emph{equal-area partition}  if $\sigma(S_i) = \frac{1}{N}$ for each $i=1,\dotsc, N$.
\end{definition}

\begin{definition}
Let  $\mathcal S = \{ S _i\}_{i=1}^N $ be an equal-area  partition of $\mathbb S^{d}$.  We say that it is a \emph{regular partition} (or, an equal area partition with bounded diameters) with constant $ K_d >0$  if  for every $i=1,\dotsc, N,$
$$\textup{diam} (S_i) \le K_d N^{-\frac{1}{d}}. $$
\end{definition}

In the case of the unit cube $[0,1)^d$, regular partitions are extremely easy to construct. Indeed, for $N=M^{d}$, one can simply take disjoint squares of side length $M^{-1}= N^{-1/d}$. The situation is  more complicated  for the sphere $\mathbb S^d$.
Nevertheless, there is an explicit upper bound on the constant $ K_d$ above.  

%%%%%%%%%%%%%%%%%%%%%%%%%%%%%% THEOREM THEOREM THEOREM
\begin{theorem}\label{t:leopardi}[Leopardi \cite{MR2582801}]   For all $ N \in \mathbb N $ there exist  regular partitions 
$  \{ S _i\}_{i=1}^N$ of         $\,\mathbb S^d$ with the constant $ K_d$  given by
\begin{equation}\label{const1}
K_d = 8 \bigg( \frac{\Omega d}{\omega} \bigg)^{\frac1{d}},
\end{equation}
where as before $\Omega$ is the $d$-dimensional Lebesgue surface measure of $\mathbb S^d$, and $\omega$ is the $(d-1)$-dimensional measure of $\mathbb S^{d-1}$. 
\end{theorem}
%%%%%%%%%%%%%%%%%%%%%%%%%%%%%% THEOREM THEOREM THEOREM

%\todo[inline]{The reference doesn't state the result as above. And, the point seems to be that $ K_d$ is bounded 
%by a constant, which is how Leopardi states his result. Perhaps state the bound here?}

Admittedly, Leopardi states the result   without the specific value of the constant $K_d$. However, it can be easily extracted from the proof, see page 9 in \cite{MR2582801}.

The history of this issue (as described in  \cite{MR2280380}) is  interesting:  Stolarsky \cite{MR0333995} asserts the existence of regular partitions of $\mathbb S^d$ for all $d\ge 2$, but offers no construction or proof of this fact. Later, Beck and Chen \cite{MR903025} quote Stolarsky, and Bourgain and Lindenstrauss \cite{MR981745} quote Beck and Chen. A complete construction of a regular partition of the sphere in arbitrary dimension was given by Feige and Schechtman \cite{MR1900615}, and Leopardi \cite{MR2582801} found an effective   value of the constant in their construction, which we quote above. 
%\begin{equation}\label{const1}
%K_d = 8 \bigg( \frac{\Omega d}{\omega} \bigg)^{\frac1{d}},
%\end{equation}
%where as before $\Omega$ is the $d$-dimensional Lebesgue surface measure of $\mathbb S^d$, and $\omega$ is the $(d-1)$-dimensional measure of $\mathbb S^{d-1}$. % It is easy to see that $\Omega/\omega \approx d^{-1/2}$, hence
%\begin{equation}\label{const2}
%K_d \approx d^{\frac1{2d}}
%\end{equation}
%with implicit constants independent of the dimension. 

\subsection{Approximating families}\label{s.af}
We approximate  an  infinite family of sets (e.g., all spherical caps, or wedges) 
by finite families.  This will facilitate the use of a union bound estimate in the next section.  

%In order to apply  the union  bound and to pass from a large deviation estimate for a single set in the collection to the whole family, one needs to reduce matters from an infinite family of sets (e.g., all spherical caps, wedges etc) to an appropriately close finite family.

\begin{definition}
Let $\mathcal S$ and  $\mathcal Q$ be two collections of subsets of $\mathbb S^d$. We say that  $\mathcal Q$ is an 
\emph{$\varepsilon$-approximating family}  (also known as $\varepsilon$-bracketing) for $\mathcal S$ if for each $S \in \mathcal S$ there exist sets $A$, $B\in \mathcal Q$ such that 
\begin{equation}
A \subset S \subset B \,\,\,\,\, \textup{ and } \,\,\, \sigma (B\setminus A) < \varepsilon.
\end{equation}
\end{definition}

It is easy to see that for any $N$-point set $Z$  in $\mathbb S^d$, if $S$, $A$, $B$ are as in the definition above then the discrepancies of $Z$ with respect to these sets satisfy 
\begin{equation}\label{e.approx}
|D(Z,S)| \le \max \{ |D(Z,A)|, |D(Z,B) | \} + \varepsilon.
\end{equation}
Hence $D_{\mathcal S} (Z)  \le D_{\mathcal Q} (Z) + \varepsilon$.
Thus for $\varepsilon \le N^{-1}$, the discrepancy with  respect to the original family is of the same order as the discrepancy with respect to the  $ \epsilon $-approximating family.

Constructions of  finite approximating families  are obvious in some cases, e.g. axis-parallel boxes in the unit cube or spherical caps: just take the same sets with rational parameters with small denominators.

For the spherical wedges, which is our case of interest, we have the following Lemma. 

%%%%%%%%%%%%%%%%%%%%%%%%%%%%%% LEMMA LEMMA LEMMA
\begin{lemma}\label{l:wedge}  For any $ 0< \varepsilon < 1$, and integers $ d \geq 1$, 
there is an approximating family $ \mathcal Q$ for the 
the collection of spherical wedges $ \{ W _{xy} \;:\; x,y \in \mathbb S ^{d}\}$  with 
\begin{equation}\label{e.afamily} 
\#  \mathcal Q    \le  {(Cd)}^{d+1} \varepsilon^{-2(d+1)}, 
\end{equation}
where $0< C \leq 82$ is an absolute constant.  
\end{lemma}
%%%%%%%%%%%%%%%%%%%%%%%%%%%%%% LEMMA LEMMA LEMMA

%%%%%%%%%%%%%%%%%%%%%%%%%%%%%% PROOF PROOF PROOF
\begin{proof}
We construct two separate families: one for interior and one for exterior approximation of the spherical wedges.
Let $\mathcal N(\varepsilon)$ be the covering number   of  $\mathbb S^d$ with respect to  the Euclidean metric, in other words, the cardinality of the smallest   set $\mathcal H_\varepsilon$ such that for each $x\in \mathbb S^d$ there exists $z\in \mathcal H_\varepsilon$ with $\| x- z\| \le \varepsilon$.

A simple volume argument (see e.g. \cite{MR2963170}) shows that
\begin{equation}
\mathcal N(\varepsilon ) \le \left( 1 + \frac2{\varepsilon} \right)^{d+1} \le \left( \frac4{\varepsilon} \right)^{d+1} .
\end{equation}
(More precise estimates can be obtained, in particular  by using $d$-dimensional, rather than $(d+1)$-dimensional volume arguments, but this will suffice for our purposes). 

We shall construct an $\varepsilon$-approximating family as follows. Start with an $\gamma$-net $\mathcal H_{\gamma}$ of size $\mathcal N( \gamma)$, where $\gamma >0$ is to be specified. For $x$, $y\in \mathbb S^d$, define the exterior enlargement and interior reduction of $W_{xy}$ as
\begin{align*}
W_{xy}^{\textup{ext}} (\gamma) & = \{ p\in \mathbb S^{d}:\, p\cdot  x  \ge -\gamma, \, p\cdot y  \le \gamma\} \bigcup \{ p\in \mathbb S^{d}:\, p\cdot  x  \le \gamma, \, p\cdot y  \ge - \gamma\} \supset W_{xy}\\
W_{xy}^{\textup{int}} (\gamma) & = \{ p\in \mathbb S^{d}:\, p\cdot  x  \ge \gamma, \, p\cdot y  \le  - \gamma\} \bigcup \{ p\in \mathbb S^{d}:\, p\cdot  x  \le - \gamma, \, p\cdot y  \ge  \gamma\} \subset W_{xy}.
\end{align*}
We claim that the collection $$ \mathcal Q =  \{ W_{xy}^{\textup{int}} (\gamma) : \, x,y \in \mathcal H_\gamma\} \cup \{ W_{xy}^{\textup{ext}} (\gamma): \, x,y \in \mathcal H_\gamma\}$$ forms an approximating family for the set of all wedges $\{  W_{xy}: \, x,y \in \mathbb S^d\}$.  Indeed, let $x'$, $y' \in \mathbb S^d$. Choose $x$, $y\in \mathcal H_\gamma$ so that $\| x-x'\| < \gamma$, $\| y - y'\| \le \gamma$.   Then it is easy to see that  $ W_{xy}^{\textup{int}}  \subset W_{x'y'} \subset W_{xy}^{\textup{ext}}$.  Indeed, e.g. if $p\cdot x' \ge 0$, then $p\cdot x = p\cdot x' - p\cdot (x'-x) \ge - \gamma$, the rest is similar. 

Moreover, it is easy to see that the normalized measure of a `tropical belt' around the equator satisfies  $$\sigma \big(\{ p\in \mathbb S^d:\, |p\cdot x| \le \gamma \} \big) \le \frac{2\gamma \omega}{\Omega}.$$ Therefore we can estimate $$ \sigma  \big(W_{xy}^{\textup{ext}} (\gamma) \setminus W_{xy}^{\textup{int}} (\gamma)\big) \le  \frac{4 \omega \gamma}{\Omega}, $$
hence we have an $\varepsilon$-approximating family with $\varepsilon = \frac{4 \omega \gamma}{\Omega}$, i.e. $\gamma = \frac{\Omega \varepsilon}{4 \omega}$.

The cardinality of this family satisfies the bound
\begin{equation}%\label{e.afamily}
\# \mathcal Q  = 2 \big(\mathcal N(\gamma)\big)^2 \le 2 \left( \frac{4}{\gamma} \right)^{2(d+1)}  =  2^{8d+9} \left( \frac{\omega}{\Omega} \right)^{2(d+1)} \cdot \varepsilon^{-2(d+1)} \le  {(Cd)}^{d+1} \varepsilon^{-2(d+1)},
\end{equation}
where $C>0$ is an absolute constant which can be taken to be, e.g. $C=82$, and we  have used the standard  fact  \eqref{e:ratio}.  

\end{proof}
%%%%%%%%%%%%%%%%%%%%%%%%%%%%%% PROOF PROOF PROOF

%\vskip1cm
%$\mathcal H_\varepsilon$ THE CONSTRUCTION OTLINE GOES HERE THE SIZE OF THE FAMILY SHOULD BE OF THE ORDER $$C_d N^{cd}.$$ The value of $C_d$ is important, $c$ not so much.

\subsection{The spherical  wedge discrepancy of jittered sampling. Proof of Theorem \ref{t.wupper}}\label{s.dwjs}
%We prove the  analog of Lemma \ref{l1} for spherical wedges

The algorithm, which we describe for the case of spherical wedges, is generic and applies to many other situations. 
We shall need the following version of the classical Chernoff--Hoeffding large deviation bound (see e.g. \cite{MR1779341,MR1697825}).
\begin{lemma}\label{hoefding}
Let $p_i \in [0,1]$, $i=1,2 ,\dotsc, m$. Consider centered independent random variables $X_i$, $i=1 ,\dotsc, m$ such that $\mathbb P (X_i = - p_i) = 1- p_i$ and $\mathbb P (X_i = 1- p_i) = p_i$. Let $X=\sum_{i=1}^m X_i$. Then for any $\lambda >0$
\begin{equation}\label{e.hoef}
\mathbb P (|X| > \lambda )  < 2 \exp \bigg( - \frac{2\lambda^2}{m} \bigg).
\end{equation}
\end{lemma}

We start with a regular partition $\{ S_i\}_{i=1}^N$ of the sphere as described in \S \ref{s.rps}, i.e. $\mathbb S^d = \cup_{i=1}^N S_i$, $\sigma (S_i \cap S_j ) = 0$ for $i \neq j$, $\sigma ( S_i ) = 1/N $, and $\textup{diam} (S_i) \le K_d N^{-1/d}$ for all $i=1 ,\dotsc, N$. 

We now construct the set $Z=\{ z_1 ,\dotsc,  z_N\}$ by choosing independent random points $z_i \in S_i$ according to the uniform distribution on $S_i$, i.e. $N\cdot \sigma\, \vline_{\, S_i}$. 

Let $\mathcal Q$ be a $1/N$-approximating family for the family $\mathcal R$ of interest, in our case the family of spherical wedges $\{ W_{xy}: x,y \in \mathbb S^d\}$. The size of this family, as discussed in \S \ref{s.af},  satisfies $\# \mathcal Q \le A_d N^{\alpha_d}$. According to \eqref{e.afamily} we may take $A_d = (Cd)^{d+1}$ and $\alpha_d = 2(d+1)$.

Consider a single set $Q \in \mathcal Q$. It is easy to see that for those $i=1,\dotsc,N$, for which $S_i \cap \partial Q = \emptyset$ ($S_i$ lies completely inside or completely outside of $Q$), the input of $z_i$ and $S_i$ to the discrepancy of $Z$ with respect to $Q$ is zero. In other words,
\begin{equation}\label{e.discsum}
D(Z, Q)  =  \frac1{N}\sum_{i: S_i \cap \partial Q \neq \emptyset } \big( \mathbf 1 _{Q} (z_i)  - N \sigma(S_i \cap Q)  \big) =  \frac1{N}  \sum_{i=1}^m X_i,
\end{equation}
where $X_i $ are exactly as in Lemma \ref{hoefding} with $p_i = N \cdot \sigma (S_i \cap Q)$, and $m \le M$, where $M$ is the maximal number of sets $S_i$ that may intersect the boundary of any element of $\mathcal Q$.

It is now straightforward to estimate $M$ for spherical wedges. Let $\varepsilon = K_d N^{-1/d}$. Since every $S_i$ has diameter at most $\varepsilon$, all the sets $S_i$ which intersect $\partial Q$ are contained in the set $\partial Q + \varepsilon B$, where $B$ is the unit ball. Recall that  $\sigma^*_d $ denotes the unnormalized Lebesgue measure on the sphere, and that we defined $\Omega = \sigma^*_d (\mathbb S^{d})$ and $\omega = \sigma^*_{d -1} (\mathbb S^{d-1})$. We then have 
\begin{equation}
M\cdot \frac{\Omega }{N} \le \sigma^*_d (\partial Q + \varepsilon B) \le \sigma^*_{d-1} (\partial Q) \cdot 2\varepsilon \le 8 K_d N^{-1/d} \omega.
\end{equation}
Hence, invoking the  diameter bounds for the regular partition \eqref{const1},   we find that 
\begin{equation}\label{e.M}
M \le \frac{8K_d \omega}{\Omega} \cdot  N^{1 - \frac1{d}}  \le  64  d^{\frac{1}{d}} \bigg( \frac{\omega }{\Omega} \bigg)^{1-\frac1{d}} N^{1-\frac1{d}}.
\end{equation}

Choosing the parameter $\lambda = ({\alpha_d} \cdot M)^\frac12 \sqrt{\log N}$, then invoking the representation \eqref{e.discsum} and the large deviation estimate \eqref{e.hoef}, we find that for any given $Q \in \mathcal Q$
\begin{equation}
\mathbb P (|D(Z, Q) | > {\lambda}/N ) = \mathbb P ( |X| > \lambda ) \le 2 N^{-2 \alpha_d}. 
\end{equation}

Since $\# Q_d \le A_d N^{\alpha_d}$,  the union bound yields
\begin{equation}
\mathbb P ( | D(Z, Q) | > \lambda/N \textup{ for at least one  } Q \in \mathcal Q) \le 2A_d N^{-\alpha_d} < 1,
\end{equation}
whenever $N> (2A_d)^{1/\alpha_d}$.  Therefore, for such $N$, there exists $Z$ such that 
\begin{equation}\label{e.upperapprox}
\sup_{Q \in \mathcal Q} | D(Z, Q)  | \le  N^{-1} (\alpha_d M)^\frac12 \sqrt{\log N} \le 8 \sqrt{\alpha_d}\, d^{\frac{1}{2d}} \bigg( \frac{\omega }{\Omega} \bigg)^{\frac12-\frac1{2d}} N^{-\frac12-\frac1{2d}} \sqrt{\log N},
\end{equation}
i.e. the discrepancy estimate of the form \eqref{e.wupper}  holds for each member of the approximating family $Q\in \mathcal Q$  with constant $ 8 \sqrt{\alpha_d}\, d^{\frac{1}{2d}} \left( \frac{\omega }{\Omega} \right)^{\frac12-\frac1{2d}}$ for $N> (2A_d)^{1/\alpha_d}$. Since this constant is greater then one, the right-hand side is greater than $\frac{1}{N}$ for all $N$. Thus according to \eqref{e.approx}, the discrepancy estimate \eqref{e.wupper} holds for all sets $W_{xy}$, $x$, $y\in \mathbb S^d$ with twice the constant.

Recalling that $\alpha_d = 2(d+1)$ and $A_d = (Cd)^{d+1}$ and the fact  that $\frac{\omega}{\Omega} \le \sqrt{\frac{d}{2\pi}}$, we find that the constant is at most $C_d = 20 d^{\frac34 +\frac1{4d}}$ whenever $N \ge 100 d$. This finishes the proof of Theorem \ref{t.wupper}. \,\,\, $\square$ \\

%\vskip2cm

%we may take i.e. $C_d =  16 \sqrt{\alpha_d}\, d^{\frac{1}{2d}} \left( \frac{\omega }{\Omega} \right)^{\frac12-\frac1{2d}}$ whenever $N> (2A_d)^{1/\alpha_d}$. 

%\textcolor{red}{{\emph{Remark:}} There are a couple more things that we may to want to check here.}

%-- I'll run over the argument to see if it gives the same thing for  the spherical cap discrepancy. The constants should be more or less the same. The behavior of these constants has never been explored, so this may be interesting to include it in the paper.

%--   In the argument,we chose the $\varepsilon$-approximating family with $\varepsilon = \frac1{N}$. We can actually try to optimize in $\varepsilon$, so that it is of the same order as the right-hand side of \eqref{e.upperapprox} and see if it gives a better result. At this point I just wanted to write down some argument that works and gives decent constants.  I will run this calculation.\\

%\subsection{Bound for uniform tessellation}

\noindent {\emph{Proof of Theorem \ref{t.natess}.}
It is  a straightforward, but tedious task to check that if $N\ge 100d$ and  $$ N> 400 d^\gamma \, \delta^{-\frac{2d}{d+1}}\,  \left( (d+1) \log (400 d^\gamma) + 2d \log \frac1{\delta} \right)^\frac{d}{d+1}, $$ where $\gamma = \frac32 - \frac1{d+1}$, one has $ 20 d^{\frac34 +\frac1{4d}} N^{-\frac12 - \frac{1}{2d}} \sqrt{\log N}
< \delta. $
Therefore,  with positive probability jittered sampling with $N$ points yields a $\delta$-uniform tessellation. It is easy to see that the right-hand side in the equation above  is bounded by $4000 d^\alpha \,\, \delta^{-2 + \frac2{d+1}}  \, \left(  1 + \log d + \log \frac1{\delta}  \right)^{\frac{d}{d+1}}  ,$
where $\alpha = \frac52 - \frac{2}{d+1}$, which proves Theorem \ref{t.natess}.

\section{ Stolarsky  principle for the wedge discrepancy.}\label{s.stol}

We now turn to the proof of Stolarsky principle for tessellations, Theorem \ref{t.stol}. Recall that  the $L^2$ norm of the function $\Delta_Z (x,y)$ for a  set $Z\subset \mathbb S^{d}$ is 
\begin{equation}\label{e.L2again}
\big\| \Delta_Z (x,y) \big\|_2^2\,\, = \,\, \int\limits_{\mathbb S^{d}}  \int\limits_{\mathbb S^{d}}  \left( \frac{1}{N} \sum_{k=1}^N {\bf 1}_{W_{xy} } (z_k)  - \sigma (W_{xy}) \right)^2 d\sigma (x) \, d\sigma(y).
\end{equation}
The proof is quite elementary in nature and conforms to a standard algorithm of many similar problems: we square out the expression above, and the cross terms yield the discrete potential energy of the interactions of points of $Z$. The idea is generally quite fruitful. Torquato \cite{To} applies this approach (both theoretically and numerically) to many questions of discrete geometric optimization, such as packings, coverings, number variance,  to recast them as energy-minimization problems. \\
% and %will try to see whether we can 
%obtain an  analogue of the Stolarsky principle (which relates the $L^2$ spherical cap discrepancy to the  sum of pairwise distances  of points in $Z$).  We have the following theorem.

\noindent {\emph{Proof of Theorem \ref{t.stol}.}} We recall that $\sigma (W_{xy}) =   d(x,y)$ and notice that we can write (up to sets of measure zero)
  $$\displaystyle{ {\bf 1}_{W_{xy} } (z_k) = {\bf 1}_{\{ \textup{sgn} (x\cdot z_k) \neq \textup{sgn} (y\cdot z_k) \}} (z_k) = \frac12 \Big(1-  \textup{sgn} (x\cdot z_k)  \cdot \textup{sgn} (y\cdot z_k)  \Big)},$$ 
  therefore, using \eqref{e.L2again}, we have 
\begin{align}
\nonumber  \int\limits_{\mathbb S^d} \!&\int\limits_{\mathbb S^{d}}   \Delta_Z (x,y)^2 d\sigma(x)\, d\sigma (y) \, =   \\
\label{term1} & =  \frac{1}{4N^2} \int\limits_{\mathbb S^d}\! \int\limits_{\mathbb S^d}  \sum_{i,j=1}^N \Big(1- \textup{sgn} (x\cdot z_i) \textup{sgn} (y\cdot z_i)  \Big) \Big(1- \textup{sgn} (x\cdot z_j)   \textup{sgn} (y\cdot z_j)  \Big)\, d\sigma (x) \, d\sigma (y)\\ 
\label{term2} & \qquad  - \frac{2}{N}  \sum_{k =1}^N \int\limits_{\mathbb S^d} \!\int\limits_{\mathbb S^d}   {\bf 1}_{W_{xy} } (z_k)  \cdot d(x,y)\,\, d\sigma (x) \, d\sigma (y)\\
\label{term3} &\qquad  +  \int\limits_{\mathbb S^{d}} \!\int\limits_{\mathbb S^{d}}  d(x,y)^2 \, d\sigma (x) \, d\sigma (y). 
\end{align}\vskip2mm
\noindent The most interesting term is the first one \eqref{term1}. Using the obvious fact that the integral  $ \int\limits_{\mathbb S^{d}} \textup{sgn} (p\cdot x) \, d\sigma (x) = 0$ for any $p \in \mathbb S^d$, we reduce this term to
\begin{align}
  \frac14 + & \frac{1}{4 N^2} \sum_{i,j=1}^N \bigg( \int\limits_{\mathbb S^d}  \textup{sgn} (x\cdot z_i)  \cdot \textup{sgn} (x\cdot z_j) \, d\sigma (x) \bigg)^2   
\label{term1a} = \frac14 + \frac{1}{ N^2 } \sum_{i,j=1}^N \bigg( \frac{1}{2} - d(z_i,z_j) \bigg)^2,
\end{align} 
where the last line is obtained as follows. Rewrite the integrand as $ \textup{sgn} (x\cdot z_i)    \textup{sgn} (x\cdot z_j) =  1 -  2 \cdot {\bf 1}_{W_{z_i z_j}} (x),$ therefore $$ \int\limits_{\mathbb S^d}  \textup{sgn} (x\cdot z_i)    \textup{sgn} (x\cdot z_j) \, d\sigma (x) = 1 - 2 \int\limits_{\mathbb S^d}   {\bf 1}_{W_{z_i z_j}} (x)\, d\sigma (x)  =  1 - {2 d(z_i,z_j)}.$$ \vskip2mm

%I haven't quite finished computing  the other two terms yet, but it is easy to see that they are independent of $N$. In particular, the second term \eqref{term2} reduces to 
\noindent We shall see that in the second term  \eqref{term2} one can easily replace the discrete average over $z_k \in Z$ by the continuous average over $p\in \mathbb S^d$, which is simpler to handle. Indeed, notice that by rotational invariance the integrand in \eqref{term2} does not depend on the particular choice of $z_k \in \mathbb S^d$. Therefore, for an arbitrary pole $p\in \mathbb S^d $ we can write
\begin{equation}\nonumber
\frac{2}{ N}  \sum_{k =1}^N \int\limits_{\mathbb S^d} \int\limits_{\mathbb S^d}   {\bf 1}_{W_{xy} } (z_k)  \cdot d(x,y)\,\, d\sigma (x) \, d\sigma (y) = 2   \int\limits_{\mathbb S^d} \int\limits_{\mathbb S^d}   {\bf 1}_{W_{xy} } (p)  \cdot d(x,y)\,\, d\sigma (x) \, d\sigma (y)
\end{equation}
Invoking rotational symmetry again, we see that the integral above may be replaced by the average over $p\in \mathbb S^d$:
\begin{align}\nonumber
& 2  \int\limits_{\mathbb S^d} \int\limits_{\mathbb S^d}    {\bf 1}_{W_{xy} } (p)  \cdot d(x,y)\,\, d\sigma (x) \, d\sigma (y)  =  2  \int\limits_{\mathbb S^d} \int\limits_{\mathbb S^d} \int\limits_{\mathbb S^d}   {\bf 1}_{W_{xy} } (p)  \cdot d(x,y)\,\, d\sigma (x) \, d\sigma (y) d\sigma(p)\\
\nonumber & =  2  \int\limits_{\mathbb S^d} \int\limits_{\mathbb S^d}  \Bigg[\,\, \int\limits_{\mathbb S^d}   {\bf 1}_{W_{xy} } (p) d\sigma(p) \Bigg]   \cdot d(x,y)\,\, d\sigma (x) \, d\sigma (y)  \,\, = \,\, 2  \int\limits_{\mathbb S^d} \int\limits_{\mathbb S^d}  d(x,y)^2 \, d\sigma (x) \, d\sigma (y),
\end{align}
thus the term has the same form as the last one \eqref{term3}. Putting this together we find that
\begin{equation}
\big\| \Delta_Z (x,y) \big\|^2_2
\, = \, \frac{1}{ N^2}  \sum_{i,j=1}^N \Big(\frac{1}2 -  d (z_i, z_j) \Big)^2 + \frac{1}{4}  -  \int\limits_{\mathbb S^d} \int\limits_{\mathbb S^d}  d(x,y)^2 \, d\sigma (x) \, d\sigma (y).
\end{equation}
Observing that  $\int\limits_{\mathbb S^d} d(x,y) d\sigma(x) = \frac12$  and hence
$$ \int\limits_{\mathbb S^d}  \int\limits_{\mathbb S^d}   \Big(\frac{1}2 -  d (x,y) \Big)^2  d\sigma (x) \, d\sigma(y) =   \int\limits_{\mathbb S^d} \int\limits_{\mathbb S^d}  d(x,y)^2 \, d\sigma (x) \, d\sigma (y)  - \frac14,$$ we arrive to the desired conclusion \eqref{e.stol}:
\begin{equation}\label{e.stol1}
\big\| \Delta_Z (x,y) \big\|^2_2
\, = \, \frac{1}{ N^2}  \sum_{i,j=1}^N \Big(\frac{1}2 -  d (z_i, z_j) \Big)^2    -  \int\limits_{\mathbb S^d} \int\limits_{\mathbb S^d}  \Big( \frac12 - d(x,y) \Big)^2 \, d\sigma (x) \, d\sigma (y).\,\,\,\,\, \square
\end{equation}

%\begin{equation}
%- \frac{4}{\pi}  \int\limits_{S_+^{d-1}} \int\limits_{S_-^{d-1}} d(x,y) \, d\sigma (x) \, d\sigma (y),
%\end{equation}
%where the hemispheres $S^{d-1}_\pm$ are defined as $S_\pm^{d-1} = \{x\in S^{d-1}:\, \textup{sign} (x\cdot p ) = \pm 1\}$ for some fixed pole $p \in S^{d-1}$.  This may be tricky to compute exactly, but should be doable.

%This theorem says that the $L^2$ discrepancy associated to the tessellation is equal to the difference between the discrete and continuous averages of the quantity $\big( 1/2 - d(x,y) \big)^2$.\\

%This seems to imply that a good tessellation (in the $L^2$ sense) could be produced by a set where all the vectors are close to being pairwise orthogonal, i.e. by a good {\bf{spherical code}}. In addition, Alex Iosevich mentioned in his talk an example of a spherical point set (due to M. Rudnev) in which many pairs are orthogonal. This is all very raw and needs further investigation.

%The problem described here arises quite naturally in {\emph{compressed sensing}}, in particular, in so-called {\emph{one-bit sensing}}. It was recently studied by Vershynin and Plan \cite{MR3164174} in this context, but they looked only at the uniform ($L^\infty$) estimates, not $L^2$, so this part is new.

%It seems likely that this $L^2$ discrepancy should satisfy bounds similar to the case of spherical caps, i.e. $N^{-\frac12-\frac1{2{(d-1)}}}$ \cite{MR762175},\cite{St}, and these considerations may indicate how to search for sets with small discrepancy.

\subsection{$L^2$ discrepancy for random tessellations} Stolarsky principle provides a very simple way to compute    the expected value of the square of the $L^2$ discrepancy. Assume that the set $Z = \{ z_1,\dotsc, z_N \} \subset \mathbb S^d$ is random and compute the expectation of $\big\| \Delta_Z (x,y) \big\|^2_2$. Obviously, for a typical point set $Z$ and a typical wedge $W_{xy}$ the discrepancy is of the order  $1/\sqrt{N}$, therefore this  expected value naturally behaves as $\mathcal O (1/N)$.  To   compute its value precisely we shall need the quantity that already arose in the computations above, namely the second moment of the geodesic distance, i.e. the expected value of the square of the geodesic distance between two random points on the sphere:
\begin{equation}
V_d = \mathbb E_{xy}\,  d(x,y)^2  = \int\limits_{\mathbb S^d} \int\limits_{\mathbb S^d}  d(x,y)^2 \, d\sigma (x) \, d\sigma (y).%= \frac{\omega_{d-2}}{\pi^2 \omega_{d-1}} \int\limits_{0}^\pi \phi^2  (\sin \phi)^{d-2} \, d\phi,
\end{equation} 
\begin{lemma}\label{t.L2random}
Let $Z = \{ z_1, ..., z_N\} \subset \mathbb S^d $ consist of $N$ i.i.d. uniformly distributed points on the sphere. Then
\begin{equation}
\mathbb E_Z \big\| \Delta_Z (x,y) \big\|^2_2 =  \frac{1}{N}  \cdot \bigg( \frac12 - {V_d} \bigg).
\end{equation}
\end{lemma}

\begin{proof}
%We use the computation \eqref{term1}-\eqref{term3} from the proof of the Stolarsky principle. 
% (as noted above $V_3 = \pi$). 
%where $\omega_{d-1}$ is the surface area of  $S^{d-1}$. 
It is obvious that ${\mathbb E_{xy} d(x,y) = \frac{1}2}$ and hence $\mathbb E_{xy} \Big( \frac12 - d(x,y) \Big)^2  = V_d - \frac14$. We use the final form of the Stolarsky principle \eqref{e.stol1} to find the value   of  $\mathbb E_Z \big\| \Delta_Z (x,y) \big\|^2_2$. We  separate  the off-diagonal and diagonal terms in the discrete part of \eqref{e.stol1} to obtain
\begin{align*}
\mathbb E_Z \big\| \Delta_Z (x,y) \big\|^2_2 & =  \frac1{N^2} \sum_{i,j=1}^N \mathbb E_{z_i,z_j} \Big(\frac12 - d(z_i,z_j) \Big)^2  - \Big(V_d - \frac14 \Big)\\
& = \frac{1 } {N^2} \cdot (N^2 -N)  \cdot  \Big(V_d - \frac14\Big) +  \frac{1}{N^2} \cdot N \cdot \frac14 - \Big(V_d - \frac14\Big)\\
& = \frac{1}{N}  \cdot \Big( \frac12 - {V_d} \Big),
\end{align*}
which finishes the proof.
%\begin{align}
%\notag  & \frac14 + \frac{1}{ N^2} \cdot (N^2 - N) \cdot \Big( \frac{1}{4} -  \mathbb E_{xy} d(x,y) + \mathbb E_{xy} \big( d(x,y) \big)^2\Big) +  \frac{1}{ N^2} \cdot N \cdot \frac{1}{4}  = {V_d}{} - \frac{V_d}{ N} + \frac{1}{2N}.
%\end{align}
%The second term \eqref{term2} yields
%\begin{align*}
%- \frac{2}{ N}  \int\limits_{S^{d-1}} \int\limits_{S^{d-1}}  \sum_{i=1}^N \mathbb E_{z_i} {\bf 1}_{W_{xy}  } (z_i)  & \cdot d(x,y)\,\, d\sigma (x) \, d\sigma (y)  =  - 2 \int\limits_{S^{d-1}} \int\limits_{S^{d-1}}  d(x,y)^2 \, d\sigma (x) = - {2V_d},
%\end{align*}
%while the third term \eqref{term3} obviously just gives $\displaystyle{{V_d}{}}$. 
\end{proof}

In the case of spherical cap discrepancy this computation is even simpler:  %similar to Lemma \ref{t.L2random} is even simpler. It yields:
\begin{lemma}
Let $Z = \{ z_1, ..., z_N\} \subset \mathbb S^d $ consist of $N$ i.i.d. uniformly distributed points on the sphere. Then
\begin{equation}
\mathbb E_Z D_{\textup{cap},L^2}^2 =  \frac{c_d U_d}{N} ,
\end{equation}
where $U_d = \mathbb E_{x,y \in \mathbb S^d}  \| x- y \| $ and $c_d$ is the constant from Theorem \ref{t.originalstol}.
\end{lemma}
Indeed, using the original Stolarsky principle, Theorem \ref{t.originalstol}, one gets
$$  \frac{1}{c_d} \mathbb E_Z D_{\textup{cap},L^2}^2 =  \frac{1}{N^2} \sum_{i,j=1}^N \mathbb E_{z_i,z_j} \| z_i - z_j \|  -  U_d  = \frac{N^2 - N}{N^2} U_d - U_d = \frac{U_d}{N}.$$

%\subsection{The second moment of the geodesic distance on the sphere}

%Finally, the last term \eqref{term3} is easily rewritten as 
Finally, we  take a closer look at the expected  value of the square of the geodesic distance $V_d$. We remark that it can be written as
\begin{equation}
V_d =  \int\limits_{\mathbb S ^{d}} \int\limits_{\mathbb S ^{d}}  d(x,y)^2 \, d\sigma (x) \, d\sigma (y) = \frac{1}{\pi^2} \cdot \frac{\omega}{\Omega} \int\limits_{0}^\pi \phi^2  (\sin \phi)^{d-1} \, d\phi,
\end{equation}
where $\omega$ is the surface area of  $\mathbb S ^{d-1}$. %For any given value of $d\ge 3$, the integrals above may be evaluated directly, although no simple closed form expression seems to be available. 
In 
Table \ref{table:t1} we list the values of $V_d$ in low dimensions.

\begin{table}[t]
\centering
\caption{The values of $V_d$ \label{table:t1}} \vskip2mm
\begin{tabular}{|c|c|c|c|c|c|}
\hline

\raisebox{-1ex}{$d$} &  \raisebox{-1ex}{$\,\, d=2 \,\,$}  &  \raisebox{-1ex}{$\,\, d=3 \,\,$} &  \raisebox{-1ex}{$\,\, d=4 \,\,$} &  \raisebox{-1ex}{$\,\, d=5 \,\,$} &  \raisebox{-1ex}{$\,\, d=6 \,\,$}    \\ [2ex]
\hline

\raisebox{-2ex}{$\,\, V_d \,\,$}   &  \raisebox{-2ex}{$\,\,\, \displaystyle{\frac12 - \frac{2}{   \pi^2} }\,\,\,$}    &  \raisebox{-2ex}{$\,\,\, \displaystyle{\frac13 - \frac{1}{ 2  \pi^2} }\,\,\,$}  &   \raisebox{-2ex}{$\,\,\,\displaystyle{\frac12 - \frac{20}{9   \pi^2} }\,\,\,$}    &  \raisebox{-2ex}{$\,\,\,\displaystyle{\frac13 - \frac{5}{ 8  \pi^2} }\,\,\,$}  &   \raisebox{-2ex}{$\,\,\,\displaystyle{\frac12 - \frac{518}{225   \pi^2} }\,\,\,$}  \\ [4ex]
 \hline
\end{tabular} \end{table}

%With a little bit of work, one should be able to  compute this integral (I haven't had the time yet :) ). In particular, if $d=3$, then $$  \int\limits_{S^{2}} \int\limits_{S^{2}}  d(x,y)^2 \, d\sigma (x) \, d\sigma (y) =  \frac{2\pi}{4\pi } \int\limits_{0}^\pi \phi^2  \sin \phi  \, d\phi = \pi. $$\\

%\vskip3mm It turns out % appears (I haven't quite  finished all the calculations yet) 
%that if one computes the $L^2$ average of the discrepancy-type function above, it indeed contains a  ``interaction" term  of the  following form:
%In any case, we thus find that 
%\begin{equation}
%\int\limits_{S^{d-1}} \int\limits_{S^{d-1}} \big( 
%\big\| \Delta_Z (x,y) \big\|^2_2%{L^2( d\sigma(x)\, d\sigma (y))} 
%\, = \, \frac{1}{\pi^2 N^2}  \sum_{i,j=1}^N \bigg(\frac{\pi}2 -  d (z_i, z_j) \bigg)^2  -  \widetilde{C_d},
%\end{equation}
%where $\widetilde{C_d}$ is a constant that depends only on the dimension $d$. 

\subsection{$L^2$ wedge  discrepancy for jittered sampling.} Stolarsky principle \eqref{e.stol1} allows one to easily prove  that jittered sampling yields optimal order of the $L^2$ wedge discrepancy.

\begin{lemma}\label{l.l2jittered}
 Let $Z=\{ z_1, ..., z_N\} \subset \mathbb S^d$, $N\in \mathbb N$, be a point set constructed by jittered sampling corresponding to a regular partition of the sphere with constant $K_d$. Then 
 \begin{equation}
 \mathbb E_Z  \big\| \Delta_Z (x,y) \big\|^2_2 \le  K_d N^{-1- \frac1{d}}.
 \end{equation}
\end{lemma}

A matching lower bound for arbitrary $N$-point sets is known for  caps and slices and can be easily generalized to wedges, see \eqref{e.sliceL2} and the discussion immediately thereafter.

\begin{proof} We notice that for $i\neq j$ we have $\mathbb E \Big(\frac{1}2 -  d (z_i, z_j) \Big)^2 = N^2 \int\limits_{S_i} \int\limits_{S_j} \Big( \frac12 - d(x,y) \Big)^2 \, d\sigma (x) \, d\sigma (y)$, while for  $i=j$ one simply gets $\frac14$. Therefore,

\begin{align*}
\mathbb E_Z  \big\| \Delta_Z (x,y) \big\|^2_2
\, &  = \, \frac{1}{ N^2}  \sum_{i,j=1}^N \mathbb E_{z_i,z_j} \Big(\frac{1}2 -  d (z_i, z_j) \Big)^2    -  \int\limits_{\mathbb S^d} \int\limits_{\mathbb S^d}  \Big( \frac12 - d(x,y) \Big)^2 \, d\sigma (x) \, d\sigma (y)\\
& =  \sum_{i=1}^N  \int\limits_{S_i} \int\limits_{S_i} \Big( d(x,y) - d^2(x,y) \Big) \, d\sigma (x) \, d\sigma (y) \le \sum_{i=1}^N K_d N^{-\frac1{d}} \cdot \frac1{N^2} = K_d N^{-1-\frac1{d}},
\end{align*}
since $d(x,y) - d^2(x,y) \le d(x,y) \le \| x - y \| \le K_d N^{-1/d}$ for $x$, $y \in S_i$. 
\end{proof}

 \end{document}